\documentclass{article}
\usepackage{csquotes}
\usepackage[utf8]{inputenc}
\usepackage{biblatex}
\usepackage[english]{babel}
\usepackage[margin=1in]{geometry}
\usepackage{amsmath,amssymb,amsfonts,amsthm,mathtools}
\usepackage{braket} 
\usepackage{xcolor}
\usepackage{tikz}
\usetikzlibrary{fit}
\usepackage{quantikz}
\usepackage[inline]{enumitem}

\usepackage{hyperref}
\usepackage[nameinlink]{cleveref}

\hypersetup{colorlinks=true,linkcolor=blue,filecolor=magenta,urlcolor=cyan}
\newtheorem{theorem}{Theorem}[section]

\newtheorem{lemma}[theorem]{Lemma}
\theoremstyle{remark}

\theoremstyle{definition}
\newtheorem{definition}{Definition}[section]
\newtheorem{prop}{Proposition}[section]

\newcommand{\R}{\mathbb{R}}
\newcommand{\C}{\mathbb{C}}
\newcommand{\A}{\mathbb{A}}
\newcommand{\Z}{\mathbb{Z}}
\newcommand{\Q}{\mathbb{Q}}
\newcommand{\N}{\mathbb{N}}
\newcommand{\sinc}[1]{\text{sinc}({#1})}
\title{Bandwidth limited functions on locally compact abelian groups}
\author{Jean-Luc Brylinski \thanks{The author had many helpful discussions with Ranee Brylinski and with Patrick Delorme. ChatGPT pointed out several typos}}
\date{September 2026}
\DeclareUnicodeCharacter{02B9}{\ensuremath{'}}
\begin{document}

\maketitle
\begin{abstract}
This mostly expository paper gives a unified treatment of many interpolation results for bandwidth-limited functions on locally compact abelian groups, based on 
reproducing kernel Hilbert spaces. It includes an abstract interpolation theorem phrased in terms of frames.
\end{abstract}

\section{Introduction}\label{sec:intro}

This mostly expository article presents several  theorems on interpolation of functions  on $]R^n$ whose Fourier transform is constrained to lie in a given measurable set of finite volume, which we call the allowable band.

Such an interpolation theorem is well-known as the Whitttaker–Shannon interpolation formula \cite{Whit} \cite{shannon} when the band is an interval in the real line. Peterson-Middleton extended that sampling theorem to parallelepipeds  of arbitrary dimension.\cite{P-M}. A further generalization to locally
compact abelian groups was proved by Kluvanek\cite{kluvanek} and by  Faridani \cite{faridani}.

In this paper we present a uniform treatment of these interpolation results based on  a simple general construction of the structure of a reproducing kernel Hilbert space (RKHS) on the space of band-limited functions.

A new  aspect here is that the band can be an arbitrary measurable subset of finite volume, not necessarily bounded. We give examples when the locally compact abelian group is $\Q_p$ or the ring of of adeles $\A$, which may be of independent interest.

We also propose an abstract sampling theorem which assumes the existence of a frame made up of coherent states.

An advantage of the  RKHS method is to allows  the construction of a feature map from a data set to a Hilbert space, a standard method in kernel based machine learning, which  has many applications, including the  representer theorem, which is well-known in Kernel machine learning \cite{S-S}.

\section{Construction of the kernel}
We denote the hermitian inner product on any $L^2(X,\mu)$ by
\begin{equation}
\langle f \mid g\rangle = \int_{X} f(x)\, \overline{g(x)}\, d\mu(x).
\end{equation}
We recall the Fourier transform on $\R^n$. For a function $f$ in $L^1(\R^n)\cap L^2(\R^n)$, its Fourier transform, denoted by $F$ or by $\mathcal F(f)$, is the function $F$ defined by

\begin{equation}
F(k)=\int_{\R^n} f(x)\, e^{-2\pi i\, k\cdot x}\, dx
\end{equation}

This an $L^2$-function and the linear map  $\mathcal F$ extends to a linear map from $L^2(\R^n)$ to itself, Parseval's theorem says that $\mathcal F$ is an isometry whose inverse is given by 

\begin{equation}
f(x)=\int_{\R^n} F(k)\, e^{2\pi i\, k\cdot x}\, dk
\end{equation}

Now we say that the (frequency) band of $f$ is contained in a measurable subset $D\subset \R^n$ if the support of the Fourier transform $F$ is contained in $D$, which means that the restriction of $F$ to $\R^n\setminus D$ is zero. We always assume that $D$ has finite Euclidean volume, but it need not be bounded. Then $1_D$ is an $L^2$ function, whose $L^2$ norm is the square root of the volume of $D$.

The condition that $ F$ is supported on $D$ means exactly that $F=1_D.\, F$ where $1_D$ is the indicator function of $D$.
Note that $1_D$ is clearly in $L^{\infty}$ without any assumption on $D$.
Now the inverse Fourier transform of the product $1_D.\, F$ is the convolution product $g_D\star f$ where $g_D$ is the inverse Fourier transform of $1_D$. This is a general feature of Fourier analysis. So we have

\begin{equation}
(g_D\star f)\, (x)=\int_{\R^n}\, g_D(y)f(x-y)dy
\end{equation}

So by applying the inverse Fourier transform $\mathcal F^{-1} $ to the equation $F=1_D.\, F$ we get the equation

\begin{equation}
f(x)=\int_{\R^n} K(x,y)\, f(y)\, dy,\qquad \text{where}\qquad K(x,y)=g_D(x-y)
\end{equation}

This is what is called a reproducing kernel equation: $f$ actually has well-defined values at all points, and the value $f(x)$ is given by a continuous linear combination of all the other values $f(y)$. This implies that the Hilbert space $\mathcal H_D$ is a reproducing kernel Hilbert space with kernel $K(x,y)$ (see for instance \cite{lan} \cite{S-S}). The integral operator

\begin{equation}
f\mapsto \int_{\R^n} K(x,y)\, f(y)\, dy
\end{equation}
is the projection from $\mathcal H$ to its closed Hilbert subspace $\mathcal H_D$.

We note that the restriction of $\mathcal F$ to $\mathcal H_D$ is an isometry onto the subspace $L^2(D)$ of $\mathcal H$, which is the space of $L^2$ functions on $D$ for the induced Lebesgue measure on $D$.
Since the Hilbert space $\mathcal H$ is an RKHS, it contains the coherent states $phi_y$ where

\begin{equation}
\phi_y(x)=K(y,x)
\end{equation}

and we have the formula for their $L^2$ scalar products:

\begin{equation}
<\phi_x|\phi_y>=K(y,x)
\end{equation}

Evaluation of a function at a point is given by the inner product with the corresponding coherent state.

\begin{equation}
f(x)=\langle f \mid \phi_x\rangle
\end{equation}

Thus evaluation of a function at $x$ is given by a continuous linear functional on $\mathcal H_D$.

We summarize these findings below

\begin{theorem} Each element of $\cal H_D$ has a unique continuous representative given by 
\begin{equation}
f(x)=\int_{R^n}\, \hat{f}(k)e^{2\pi i x.k}dk
    \end{equation}
Then $\mathcal H_D$ is an RKHS with reproducing kernel $K(x,y)$. We have $||f||_{\infty}\leq V^{1/2}||f||_2.$

\end{theorem}

We have proved all these facts,,except for the last one, which follows from Cauchy-Schwartz.

One can think of the coherent state $\phi_x$ as a smoothing  of the Dirac delta function centered at $x$. This kind of smoothing is often used in data analysis and in machine learning. In fact, a central tool in machine learning is the \underline{feature map} $\Phi:\R^n\to\mathcal H_D$ which gives a physical representation of $\R^n$, hence of any dataset inside $\R^n$, inside the Hilbert space $\mathcal H$.

\begin{equation}
\Phi(x)=\phi_x
\end{equation}

The topological properties of the feature map (Lipschitz, smoothness, etc) and  the quality of the representation of he dataset can be  decoded from the properties of the kernel $K(x,y)$. In  our situation we have
 \begin{theorem}
The feature map is an injective continuous map from $\R^n$ to $\mathcal H_D$
If further the function $||k||$ is square-summable on $D$ then $\Phi $ is Lipschitz with Lipschitz   constant $2\pi \int_D\, ||k||^2 dk$.

 \end{theorem}

\begin{proof}
We have
\begin{equation}
  \|\Phi(x)-\Phi(y)\|_2^2
  = \int_D \left|e^{-2\pi i\, x\cdot k}-e^{-2\pi i\, y\cdot k}\right|^2\,dk
  = 2V - 2\Re\bigl(g_D(x-y)\bigr).
\end{equation}
Since $g_D$ is continuous and $g_D(0)=V$, we see that $\Phi$ is continuous. Furthermore, if $x\neq y$ the integrand only vanishes on the hyperplane perpendicular to $x-y$, which has measure $0$, so the integral is a positive number; hence $\Phi(x)\neq \Phi(y)$.

The integrand is bounded by $2\pi\, |(x-y)\cdot k|$, since a chord joining two points of a circle is shorter than an arc of the circle. This in turn is bounded by $2\pi\,\|x-y\|\,\|k\|$. So we have
\begin{equation}
  \|\Phi(x)-\Phi(y)\|_2^2 \le 2\pi\,\Bigl(\int_D \|k\|^2\,dk\Bigr)\,\|x-y\|.
\end{equation}
so that $\Phi$ is Lipschitz provided $\int_D \|k\|^2\,dk$ is finite.
\end{proof}

More generally, if $\int_D \|k\|^{2m}\,dk$ is finite, $\Phi$ is of class $C^m$. The integral representation
\begin{equation}
  f(x)=\int_D e^{2\pi i\, x\cdot k}\,\hat f(k)\,dk
\end{equation}
shows that $f$ extends to an entire analytic function on $\C^n$.

If the subset $D$ is contained in a ball of radius $R$, then we have the following inequality (by Cauchy-Schwarz):

\begin{equation}
  |f(a+ib)|\le V^{1/2}e^{2\pi R\,\|b\|}.
\end{equation}

which is a simple case of the Paley-Wiener theory.
\vskip .12 in
A number of optimizing problems for functions on datasets have a solution which is a finite linear combination of coherent states.

Also coherent states in quantum optics, also known as Glauber states, are used to model laser light.

We note that the kernel $K(x,y)$ is translation-invariant, reflecting the fact that $\mathcal H_D$ is a translation-invariant subspace of $\mathcal H$.

We thus have the relation between coherent states: $\phi_x(y)=\phi_0(y-x)$.

\section{The case of one-dimensional signals}

Take for the band $D$ a closed interval $[a,b]$. The inverse Fourier transform $g_D$ of the indicator function $1_{[a,b]}$ is

\begin{align*}
g_D(x)
&= \int_{a}^{b} e^{2\pi ikx}\, dk\\
&= -i(2\pi)^{-1}x^{-1}\bigl(e^{2\pi ibx}-e^{2\pi iax}\bigr)\\
&= \pi^{-1}x^{-1}e^{\pi i(a+b)x}\, \sin\left(\frac{(b-a)\pi x}{2}\right)\\
&= (b-a) e^{\pi i(a+b)x}\, \sinc{\frac{(b-a)\pi x}{2}}
\end{align*}

Here the signal $\sinc(x)=\frac{\sin x}{x}$ denotes a basic signal in Fourier theory.

So the RKHS $\mathcal H_D$ has the kernel

\begin{equation}
K(x,y)=(b-a) e^{\pi i(a+b)(x-y)}\, \sinc{\frac{(b-a)\pi (x-y)}{2}}
\end{equation}

    The most interesting case for signal processing is $a=-b$, since then we have  simply
    $g_D(x)=2b\, \sinc{bx}$.

    The coherent states are then
    
\begin{equation}
\phi_x(y)=K(x,y)=2b\,\sinc{\bigl(b\pi (x-y)\bigr)}
\end{equation}

    Now it is not hard to construct an ON basis for the Hilbert space $\mathcal H_D$ which is made up of coherent states. 
    
    \begin{prop}
    The coherent  states $(2b)^{-1}\phi_{(2b)^{-1}n}$ form an orthonormal basis of the Hilbert space $\mathcal H_D$.
    \end{prop}

    \begin{proof}

The inner products of two of these coherent states is
\begin{equation}
\langle \phi_{(2b)^{-1}n} \mid \phi_{(2b)^{-1}m}\rangle
= K\bigl((2b)^{-1}n,\,(2b)^{-1}m\bigr)
= 2b\,\sinc{\pi(n-m)}
= 2b\,\delta_{mn}.
\end{equation}

Hence the $(2b)^{-1}\phi_{(2b)^{-1}n}$ form an orthonormal system in $\mathcal H_D$. It remains to show that this ON system is complete. For this, we look at the images of the $\phi_{(2b)^{-1}n}$ under Fourier transform. The Fourier transform of $\phi_0=g_D$ is the indicator function $1_D$. Now each $\phi_{(2b)^{-1}n}$ is the translate of $\phi_0$ by $(2b)^{-1}n$, and we have the classical result:

\begin{lemma}
For $f\in \mathcal H=L^2(\R)$ with Fourier transform $\hat f(k)$, the Fourier transform of the function $f(x-s)$ is $e^{-2\pi i s k}\, \hat f(k)$.
\end{lemma}

So the Fourier transform of $\phi_{(2b)^{-1}n}(x)=\phi_0\bigl(x-(2b)^{-1}n\bigr)$ is
\[
 e^{-\pi i (n/b)k}\, 1_D(k).
\]
These plane waves form the usual ON basis of $L^2([-b,b])$.
\end{proof}

\begin{equation}
 f=\sum_{n\in\mathbb{Z}}\, \langle f \mid \phi_{b^{-1}n}\rangle\, \phi_{b^{-1}n}.
\end{equation}

Evaluating this equation at $x$ we get
\begin{align*}
 f(x)&=\sum_{n\in\mathbb{Z}}\, f\bigl(b^{-1}n\bigr)\, \phi_{b^{-1}n}(x)\\
 &=2b\sum_{n\in\mathbb{Z}}\, \sinc{\bigl(b\pi x-n\pi\bigr)}\, f\bigl(b^{-1}n\bigr).
\end{align*}

This is exactly the Whittaker--Shannon interpolation formula \cite{Whit}, showing that a function in the band-limited Hilbert space $\mathcal H_D$ can be completely recovered from its values on the two-sided arithmetic progression $n\in\mathbb{Z}$.

\section{The spherical case}

If the domain $D$ is invariant under the rotation group $O(n)$, then its inverse Fourier transform $g_D$ has the same invariance. We will consider the case $D=B_r$, the solid ball of radius $r$ centered at the origin. We then have from \cite{gelfand}

\begin{equation}
g_D(x)=r^{n/2}\,\lVert x\rVert^{-n/2}\, J_{n/2}\bigl(2\pi r\,\lVert x\rVert\bigr)
\end{equation}
where $J_{n/2}$ is a Bessel function of the first kind. It is the unique solution $y(x)$ of the differential equation

\begin{equation}
x^2\, \frac{d^2y}{dx^2}+x\,\frac{dy}{dx}+\left(x^2-\frac{n^2}{4}\right)y=0.
\end{equation}

which is bounded at the origin $x=0$ and has leading term $\frac{\vert\vert \vert\vert{n/2}}{\pi^{n/2}\Gamma(n/2+1)}$ in its  expansion there in powers of $x^{1/2}$.

Of course $J_{n/2}$ is in $L^2$ at infinity, as it behaves like the product of $x^{-1/2}$ with a sine wave.

\section{The case of a parallelepiped}

Assume that $D$ is a parallelepiped centered at the origin which is a fundamental domain for the lattice spanned by vectors $v_1,\cdots,v_n$.

 Then we have the equation
\begin{equation}
g_D(x)=V\, \prod_{j=1}^{n}\, \sinc{\pi x\cdot v_j}
\end{equation}
where $V$ is the volume of $D$.

To interpolate a function $f$ in $\mathcal H_D$, we introduce the dual basis $u_1,\cdots,u_n$ and the lattice $L$ spanned over $\mathbb{Z}$ by the $u_j$'s.

Hence we get the interpolation formula

\begin{equation}
f(x)=\sum_{y\in L} \prod_{j=1}^{n} \sinc{\pi\, (x-y)\cdot v_j}\, f(y)
\end{equation}
which recovers the Petersen–Middleton theorem \cite{P-M}.

\section{Locally compact abelian groups}  Our methods apply with little change to the situation where of $L^2(G)$ for an abelian locally compact group $G$ which is countable at infinity.
The dual group $\hat G$ is the group of characters $\chi:G\to\C^*$ equipped with the topology of uniform convergence on compact subsets. $\hat G$ is also a locally compact abelian group which is countable at infinity.Then we have a Fourier transform $\mathcal F$ from $L^2(G)$ to $L^2(\hat G)$ defined by the equation
\begin{equation}
\hat f(\chi)=\int_Gf(x) \overline{\chi(x)}d\mu(x)
\end{equation}

Here $\mu$ is a choice  of a Haar measure on $G$. Then there is a unique choice of a normalization $\nu$ of a  Haar measure on $\hat G$ such that the inverse of $\mathcal F$ maps a function $g(\chi)$ to the function
$x\mapsto\int_{\hat G}\, \chi(x)^{-1}g(\chi)d\nu(\chi)$.

Then for a compact subset $D$ in $\hat G$ we can define the Hilbert subspace $\mathcal H_D$ of $L^2(G)$ as the space of functions whose Fourier transform is supported on $D$. Then $\mathcal H_D$ is an RKHS with reproducing kernel $K(x,y)=g_D(xy^{-1})$ where $g_D$ is the inverse Fourier transform of the indicator function $1_D$.

To get an interpolation formula, we take $D$ to be a fundamental domain of some discrete co-compact subgroup $\Lambda$ of $\hat G$, assuming that such a subgroup can be found.. Then the dual subgroup is the subgroup $L$ of $G$ comprised  of the $x\in G$ such that $\chi(x)=1$ for all $\chi\in\Lambda$.

Then we have 

\begin{theorem}
The $V^{-1}\phi_{\lambda}$ for $\lambda\in L$ form an ON basis of the RKHS $\mathcal H_D$ and we have
\begin{equation}
f(x)=\sum_{\lambda \in L}\, g_D(x-\lambda)f(\lambda)
\end{equation}
where the series converges uniformly on $G$.
\end{theorem}

\section{The p-adic case}
We normalize the Haar measure $\mu$ on $G=\mathbb{Q}_p$ by the condition $\mu(\mathbb{Z}_p)=1$.

Recall that for $G=\mathbb{Q}_p$ one can identify the dual group $\widehat{\mathbb{Q}}_p$ with $G$ using the character $\psi:\mathbb{Q}_p\to\mathbb{C}^*$ defined by $\psi(x)=e^{2\pi i x_0}$, where $x_0$ is any rational number such that $x-x_0$ has valuation $\ge 0$. Then the group homomorphism $\Q_p\to \widehat{\Q_p}$ which sends $y$ to the character $x\mapsto\psi(xy)$ is a homeomorphism.

Now take the domain $D$ to be $D=p^m\Z_p$. Then the Fourier transform of $1_D$ is $p^{-m}\dot 1_{D^{\perp}}$, where $D^{\perp}=p^{-m}\Z_p$, and thus we have
\begin{equation}
K(x,y)=p^m\,\mathbf{1}_{p^m\mathbb{Z}_p}(x-y).
\end{equation}
Picking a discrete subset $S\subset \mathbb{Q}_p$ of representatives for the quotient $\mathbb{Q}_p/p^m\mathbb{Z}_p$, we see that the rescaled coherent states $p^{m/2}\phi_s=p^{m/2}\,\mathbf{1}_{s+p^m\mathbb{Z}_p}$ form an orthonormal basis of $L^2(D)$ as $s$ varies over $S$.

Then we have the interpolation formula.
\begin{prop}
\begin{equation}
f(x)=\sum_{s\in S}\,\mathbf{1}_{s+p^m\mathbb{Z}_p}(x)\, f(s).
\end{equation}
\end{prop}

\section{The adelic case}

A similar construction applies to the ring of adeles $\A$, which contains $\Q$ as a discrete subgroup.

Recall that
\begin{equation}
\A=\R\times\prod_p\,\Q_p,
\end{equation}
where the [product is ] the restricted product: an adele $(x_{\infty},x_2,x_3,\dots)$ must have all but finitely many of its components $x_p$ belonging to $\Z_p$.

We can then put on $\A$ the Haar measure which is the product of the measure $dx$ on $\R$ and the normalized Haar measures on $\Q_p$.

The group $\Q$ embeds diagonally as a subgroup of $\A$, and it is discrete by the product formula in number theory.

Then $D=[-1/2,1/2]\times \prod_p \Z_p$ is a fundamental domain for $\A/\Q$, and $D$ has volume $1$.

The Fourier transform $g_D$ of $1_D$ factors as
\begin{equation}
g_D(x)=\sinc{\pi x_{\infty}}\,\prod_p \mathbf{1}_{\Z_p}(x_p),
\end{equation}
and the reproducing kernel $K(x,y)=g_D(x-y)$ factorizes similarly.

Then we have the interpolation formula
\begin{equation}
f(x)=\sum_{y\in \mathbb{Q}}\, \sinc{\pi(x_{\infty}-y_{\infty})}\,\prod_p \mathbf{1}_{\Z_p}(x_p-y_p)\, f(y),
\end{equation}
which expresses $f(x)$ as an infinite sum of the values $f(y)$ for $y\in \Q$.

\section{An abstract interpolation theorem}

Working in the context of a locally compact abelian group $G$, we can state
\begin{theorem}
Assume there is a sequence $x_1,x_2,\dots$ in $G$ such that the corresponding coherent states $\phi_{x_j}$ form a Riesz basis of the RKHS $\mathcal H_D$. Let $u_1,u_2,\dots$ be the dual Riesz basis, so that we have
$<\phi_j\vert u_k>=\delta_{jk}$.
Then for any $f\in\mathcal H_D$ we have
\begin{equation}
f(y)=\sum_j\,K(x_j,y)\, f(u_j).
\end{equation}
\end{theorem}
\vskip .12 in
Recall that a Riesz basis of $\mathcal H_D$ is a topological basis which is the transform of an orthonormal basis by an invertible   bounded  operator.
\vskip .12 in
It may be worth exploring if a Riesz basis of this type exists in the case $D$ is a solid sphere. This is equivalent under Fourier transform to asking whether $L^2(D,\mu_{D})$ has a Riesz basis made up of plane waves; such a Riesz basis is called an exponential Riesz basis. It is proved in \cite{I-K-T}  that exponential Riesz bases do not exist in case $D$ has at last one point of vanishing Gaussian curvature. Very recently \cite{K-N-O} proves that exponential Riesz bases do not exist for balls in $\R^n$ for $n\ge 2$.

However there is a more general notion, namely that of frames in a Hilbert space.

\begin{definition}
A sequence $\phi_1,\phi_2,\dots$ of vectors in a Hilbert space $\mathcal E$ is a frame if there exist constants $0<A\le B<\infty$ such that
\begin{equation}
A\Vert f\Vert^2\le \sum_j \bigl|\langle f \mid \phi_j\rangle\bigr|^2\le B\lVert f\Vert^2
\end{equation}
for each vector $f$ in $\mathcal E$.
\end{definition}

Now consider a sequence $y_1,y_2,\dots$ in $G$ and the corresponding elements $\phi_j=\phi_{y_j}$.

\begin{theorem}
Assume the sequence $y_1,y_2,\dots$ satisfies the above bounds so that the $\phi_j$ form a frame in $\mathcal H_D$ with constants $A$ and $B$. Then
\begin{equation}
S(f)=\sum_j f(y_j)\,\phi_j
\end{equation}
defines an invertible bounded positive operator on $\mathcal H_D$, and for all $f\in\mathcal H_D$ we have
\begin{equation}\label{uniform}
f(x)=\sum_j f(y_j)\,\bigl(S^{-1}\phi_j\bigr)(x).
\end{equation}
The first series converges in norm, and the second series converges absolutely  uniformly on compact subsets of $G$.
\end{theorem}

\begin{proof}
Let $C\colon\mathcal H_D\to l^2(\N)$ be the sampling map $C(f)=(f(y_j))_j$. Then $S=C^*C$ is bounded and positive, and the frame bounds imply the inequalities $AI\le S\le BI$ between self-asdjoint operators. Hence $S$ is invertible with bounded inverse, and the reconstruction formula follows.

For the second series \eqref{uniform} we have

\begin{align*}
\sum_j\vert u_j(x))\vert ^2&=\sum_j\vert<\phi_x\vert u_j>\vert ^2\\
&=\sum_j\vert<\phi_x\vert S^{-1}(\phi_j)>\vert^2\\
&=\sum_j\vert S^{-1}\phi_x(y_j)\vert^2t\\
&=\vert\vert CS^{-1}\phi_x\vert\vert^2S
\end{align*}
which is bounded by  $\vert\vert CS^{-1}\vert\vert^2\dot vert\vert\phi_x\vert\vert^2$, hence by $B^{1/2}A^{-1}\vert\vert\phi_x\vert\vert^2$

Then by Cauchy-Schwartz we have

\begin{equation}
\vert\sum_j\,f(v_j)u_j(x)\vert^2\leq\bigl(\sum_j\vert gf(y_j)\vert^2\bigr)\cdot\bigl( \sum_j\vert u_j(x)\vert^2\leq b^{1/2}A^{-1}\vert\vert \phi_x\vert\vert^2
\end{equation}
\end{proof}

Since $vert\vert \phi_x\vert\vert$ is a continuous function of $x\in G$, it follows that the interpolation series
$\sum_j\,f(v_j)u_j(x)$ converges absolutely,uniformly on compact subsets of $G$.

It seems worth asking if frames consisting of coherent states can be found for domains of interest
\printbibliography

@article{P-M,
  title={Sampling and reconstruction of wave-number-limited functions in N-dimensional Euclidean spaces},
  author={Petersen, Daniel P and Middleton, David},
  journal={Information and control},
  volume={5},
  number={4},
  pages={279--323},
  year={1962},
  publisher={Elsevier}
}

@book{lan,
  author  = {N. P.Landsman},
  title   = {Mathematical Topics Between Classical and Quantum},
  publisher = {Springer Verlag},
  year    = {1998},}

@book{S-S,
  author  = {B. Scholkopf and A. Smola},
  title   = {Learning with Kernels: Support Vector Machines, Regularization, Optimization, and Beyond},
  publisher = {MIT Press},
  year    = {2001},}

@book{Whit,
  author  = {J.M. Whittaker},
  title   = {Interpolatory Functibn Theory},
  publisher = {Cambridge Press},
  year    = {1935},}

@book{gelfand,
  title={Generalized Functions: Applications of harmonic analysis, by IM Gelʹfand and N. Ya. Vilenkin, translated by A. Feinstein},
  author={Gelfand, Izrail Moiseevich and Shilov, Georgi Evgenevich},
  volume={4},
  year={1964},
  publisher={Academic Press}
}

@article{kluvanek,
title= {Sampling theorem in abstract harmonic analysis},
author={I. Kluvanek},
journal={Matematitcko-fyzikalny Casopsis},
volume={15},
year={1965},
pages={43-48}
}

@article{faridani,
author={A. Faridani},
title={A generalized sampling theorem for locally compact abelian groups},
journal={Mathematics of computation 63.207},
volume={63.207},
year={1994},
pages={307-327}
}

@article{shannon,
author={Shannon, Claude E},
title={Communication in the presence of noise},
jourtnal={Proceedings of the IRE},
volume={37},
year={1949},
pages={10-21}
}

@article{K-N-O,
  title={A set with no Riesz basis of exponentials},
  author={Kozma, Gady and Nitzan, Shahaf and Olevskii, Alexander},
  journal={Revista matem{\'a}tica iberoamericana},
  volume={39},
  number={6},
  year={2023},
  publisher={Revista Matem{\'a}tica Iberoamericana}
}

@article{I-K-T,
author={Iosevich  Alex, Nets Hawk Katz, and Terry Tao},
title={Convex bodies with a point of curvature do not have Fourier bases},
journal={American journal of mathematics},
volume={123},
year={2001},
pages={115-120}
}
\end{document}